\documentclass[12pt, reqno]{amsart}
\usepackage{amsmath, amsthm, amscd, amsfonts, amssymb, graphicx, color, mathrsfs}
\usepackage[bookmarksnumbered, colorlinks, plainpages]{hyperref}
\hypersetup{colorlinks=true,linkcolor=red, anchorcolor=green, citecolor=cyan, urlcolor=red, filecolor=magenta, pdftoolbar=true}
\usepackage{hyperref}
\newtheorem{theorem}{Theorem}[section]
\newtheorem{lemma}[theorem]{Lemma}

\newtheorem{corollary}[theorem]{Corollary}
\theoremstyle{definition}
\newtheorem{definition}[theorem]{Definition}

\theoremstyle{remark}

\numberwithin{equation}{section}

\begin{document}

\setcounter{page}{1}

\title[Non-local time problem for the Rayleigh--Stokes ...]{Non-local time problem for the Rayleigh--Stokes type fractional equations
}

\author[R.  Ashurov, Y. Fayziyev and  N. Khushvaktov]{Ravshan Ashurov, Yusuf Fayziyev and Nuriddin Khushvaktov}

\address{\textcolor[rgb]{0.00,0.00,0.84}{Ravshan Ashurov:
V. I. Romanovskiy Institute of Mathematics, 
Uzbekistan Academy of Sciences,
University Street 9, 
Tashkent 100174, 
Uzbekistan
\\
Central Asian University, 
Mirzo Ulugbek district, 
Tashkent, 111221, 
Uzbekistan
}}
\email{\textcolor[rgb]{0.00,0.00,0.84}{ashurov@gmail.com}}

\address{\textcolor[rgb]{0.00,0.00,0.84}{Yusuf Fayziyev:
National University of Uzbekistan,
University Street 4, 
Tashkent 100174, 
Uzbekistan
}}
\email{\textcolor[rgb]{0.00,0.00,0.84}{fayziev.yusuf@mail.ru}}

\address{\textcolor[rgb]{0.00,0.00,0.84}{Nuriddin Khushvaktov:
Tashkent International University 
of Financial Management and Technologies, 
Amir Temur Shoh Street 15,
Tashkent 100047, 
Uzbekistan
}}
\email{\textcolor[rgb]{0.00,0.00,0.84}{nuriddinh@gmail.com; n.xushvaqtov@tift.uz}}

\thanks{All authors contributed equally to the manuscript and read and approved the final manuscript.}

\let\thefootnote\relax\footnote{$^{*}$Ravshan Ashurov}



\begin{abstract} Despite the growing interest in fractional generalizations of classical fluid dynamics equations, the fractional Rayleigh--Stokes problem has previously been studied almost exclusively using the Riemann--Liouville fractional derivative. To the authors' knowledge, an explicit analytical form of the solution for the Caputo derivative case has not been established in the literature, and before this work, no systematic study of the existence, uniqueness, or regularity properties of this formulation has been conducted.

In this paper, we fill this gap by considering the Rayleigh--Stokes equation with the Caputo fractional time derivative of order $\rho \in (0, \, 1)$. Using the Laplace transform and Fourier methods, as well as special functions, we perform a rigorous well-posedness analysis of the corresponding initial boundary-value, non-local, and backward problems.  
\end{abstract} \maketitle
\tableofcontents

\section{Introduction}

The Rayleigh--Stokes problem (also known as the Stokes first problem with an additional inertial term) describes the unsteady flow of a viscous incompressible fluid over an infinite flat plate that is impulsively started or oscillated in its own plane. This classical model has important applications in various engineering and physical contexts, such as boundary-layer control, channel start flows, and certain rheological processes involving viscoelastic fluids \cite{AshR96}.

The incorporation of fractional time derivatives into the Rayleigh--Stokes equation enables the modeling of anomalous diffusion and memory effects characteristic of complex fluids with non-local viscoelastic properties, porous media, or materials exhibiting power-law relaxation. A typical fractional Rayleigh--Stokes equation has the form
\[
\frac{\partial u}{\partial t} = \nu \left(1 + \tau \partial_{t}^{\rho} \right) \frac{\partial^2 u}{\partial x^2} + f(t),
\]
where $\nu$ is the kinematic viscosity, $\tau$ is a relaxation parameter, and $\partial_t^{\rho}$ denotes a Riemann--Liouville fractional derivative of order $0 < \rho < 1$ (sub-diffusive regime) \cite{AshR96, AshR200}.

In recent years, fractional-derivative models of the Rayleigh--Stokes equation have been developed, and analytical and numerical solutions for generalized second-grade fluids have been investigated \cite{AshR99,AshR220,AshR230}. In some studies, inverse problems for the Rayleigh--Stokes equation, including the identification of the source term, have been considered using Tikhonov regularization and other stabilization techniques \cite{AshR240,AshR250}. In addition, several works are devoted to the reconstruction of initial conditions under random noise and the development of the corresponding regularization algorithms \cite{AshR280}.

The backward and periodic problems for the
Rayleigh--Stokes equations were considered in \cite{AshR260, AshR270, AshR98}, while a non-local problem for the fractional Rayleigh--Stokes equation was studied in \cite{AshR99}. Let us note one more work \cite{AshR275}, where for the Rayleigh--Stokes problem, an inverse problem of determining the order of the fractional derivative was posed and investigated.

If $\dfrac{\partial}{\partial t}$ is replaced by $D_{t}^{\rho}$ with $\rho \in (0, \, 1)$ the equation may lead to fractional models of the Barenblatt--Zheltov-Kochina type.
The forward and inverse problems for such equations were considered in \cite{AshR310}, and non-local problems in time were studied in \cite{AshR320}. These results show that the choice of the fractional operator has a significant influence on the stability and uniqueness of the solution.

The presence of the classical first-order time derivative $\dfrac{\partial u}{\partial t}$ on the left-hand side implies that the solution $u(x,t)$ is assumed to be (at least) once differentiable with respect to time in the classical sense. This regularity assumption makes the use of the Caputo fractional derivative particularly natural and convenient. Specifically, the Caputo derivative $D_t^{\rho} u(t)$ is defined in such a way that it coincides with the classical derivative of order $n = 1$ applied to the fractional integral of $u$, which requires existence derivates of the corresponding order of integers. In contrast, the Riemann--Liouville derivative $\partial_t^{\rho} u(t)$ does not impose such a requirement on the classical differentiability of $u$ and is often preferred in models where only weaker regularity is expected.

Thus, in the context of the Rayleigh--Stokes equation, where the classical first derivative is explicitly present, and its existence is postulated, replacing the Riemann--Liouville derivative with the Caputo derivative is fully justified and leads to a mathematically consistent formulation with physically interpretable initial conditions (e.g., initial velocity profile $u(x,0)$).

In the present paper, we study the backward and non-local problems for the Rayleigh--Stokes equation involving the Caputo derivative. The main purpose of the paper is to obtain conditions that ensure the well-posedness of initial boundary-value, non-local, and backward problems.

Let us start with a few definitions to frame the problem.

In a Hilbert space, the operator $-\dfrac{\partial^2}{\partial x^2}$ with boundary conditions is a self-adjoint operator. In order to consider the general case, we consider an abstract self-adjoint operator $A$.
Let $H$ be a separable Hilbert space, and let $A: H \to H$ be a
self-adjoint, positive, and unbounded operator with the domain $D(A)$.
We denote by $\{v_k\}$ a complete orthonormal system consisting of the eigenfunctions of $A$ and by $\{\lambda_k\}$ the corresponding positive eigenvalues. These eigenvalues may be arranged in a non-decreasing sequence such that
\[
0<\lambda_1 \le \lambda_2 \le \cdots \to +\infty .
\]

Consider a vector-valued function $h(t)$ defined in $[0,+\infty)$ and taking values in the Hilbert space $H$. For $\rho\in(0,1)$, the Caputo fractional derivative (see, e.g., \cite{AshR100}) and Riemann--Liouville (see, e.g., \cite{AshR102}) fractional derivative of order $\rho$ are defined by
\[
D_{t}^{\rho} h(t)=\dfrac{1}{\Gamma(1-\rho)} \int\limits\limits _{0}^{t} \dfrac{h'(\xi)}{(t-\xi)^{\rho}} d\xi, \quad
\partial_{t}^{\rho} h(t)=\dfrac{1}{\Gamma(1-\rho)} \dfrac{d}{dt} \int\limits\limits _{0}^{t} \dfrac{h(\xi)}{(t-\xi)^{\rho}} d\xi,
\quad t>0,
\]
provided that the right-hand sides exist.

Let $\rho \in (0,1)$ and $C((a, b); \, H)$ stand for a set of continuous functions $u(t)$ of $t \in (a, b)$ with values in $H$.

Let us consider the following Cauchy problem and the non-local problems:
\begin{equation}\label{Rav371}
\left\{
\begin{array}{ll}
D_{t} u(t)+A(1+\gamma D_{t}^{\rho}) u(t)=f(t), \quad \gamma>0, \quad 0< t \leq T, \quad D_t=\dfrac{d}{dt};\\
u(0)=\phi,
\end{array}
\right.
\end{equation}
\begin{equation}\label{Rav1}
\left\{
\begin{array}{ll}
D_{t} u(t)+A(1+\gamma D_{t}^{\rho})u(t)=f(t), \quad 0<t \leq T; \\
u(T)=\alpha u(0)+\varphi,
\end{array}
\right.
\end{equation}
where $\phi, \varphi \in H$ are  given vectors, $f \in C([0, \, T]; \, H)$ is a given function. The problem (\ref{Rav371}) is also known as \emph{the forward problem}.
The parameter
$\gamma>0$ is a constant and $\alpha$ is equal to $0$ or $1$. If $\alpha=1$, then the problem (\ref{Rav1}) is called \emph{the non-local problem}; if $\alpha=0$, it is called \emph{the backward problem}. Both cases will be considered.

Let us provide a definition for the solutions of the forward and non-local problems.
\begin{definition}\label{Rav10}
A function $u(t) \in C([0,T]; \, H)$ is called a solution of the forward, or non-local problem for the Rayleigh–Stokes equation if $D_t u(t)$, $Au(t)$, and $AD_t^\rho u(t)$ belong to $C((0,T]; \, H)$, and $u(t)$ satisfies the conditions \eqref{Rav371} for the forward problem or \eqref{Rav1} for non-local  problem, respectively.
\end{definition}

In Section 2 (Preliminaries), we introduce the order and domain of definition of the operator $A$, the Laplace transform and its inverse, and the Laplace transform of the Caputo fractional derivative. We also study the Cauchy problem for the homogeneous Rayleigh--Stokes ordinary fractional differential equation. In Section 3 (Forward Problem), we consider the Cauchy problem for the Rayleigh--Stokes equation. It should be emphasized that this specific problem has not been investigated previously.
In Section 4 (Auxiliary Problem), Problem \eqref{Rav1} is decomposed into two auxiliary problems for the case $\alpha = 1$. The first auxiliary problem coincides with the case studied in the previous section under zero initial conditions. The second leads to a non-local problem that is analyzed separately.
Finally, in Section 5, the backward problem for Problem \eqref{Rav1} is investigated for case $\alpha = 0$.

\section{Preliminaries}
In this section, we utilize fractional powers of the operator $A$ to construct a corresponding Hilbert space. We also recall the Laplace transform and key properties of the functions $A_{\rho}(\lambda,t)$ and $B_{\rho}(\lambda,t)$ introduced in \cite{AshR96}, as these will be essential in the sequel.

Let $\tau$ be an arbitrary real parameter. The operator $A$ acting on $H$ admits a fractional power, which we define as follows (the positivity of $A$ implies $\lambda_k > 0$ for all $k$):
\[A^\tau h = \sum_{k=1}^{\infty} \lambda_k^\tau h_k v_k,\]
where the numbers 
\[h_k = (h,v_k)\]
are the Fourier coefficients of $h \in H$ with respect to the orthonormal basis $\{v_k\}$.

The domain of $A^\tau$ consists of all vectors $h \in H$ such that $A^\tau h \in H$; that is,
\[D(A^\tau) = \{ h \in H : \sum_{k=1}^{\infty} \lambda_k^{2\tau} |h_k|^2 < \infty \}.\]

For the elements of $D(A^\tau)$, we introduce the norm
\[\|h\|_\tau^2 = \sum_{k=1}^{\infty} \lambda_k^{2\tau} |h_k|^2 = \|A^\tau h\|^2.\]
Endowed with this norm, the space $D(A^\tau)$ becomes a Hilbert space.

Since  $\lambda_k>0$ for all $k$, there exists a constant $C_\tau >0$, depending on $\tau$, such that  the following estimate is valid
\begin{equation}\label{Rav570}
\|h\|\leq C_\tau \|h\|_\tau,\,\, \tau>0.
\end{equation}

The Laplace transform of a function $y(t)$ is defined by (see \cite{AshR102}, p.18)
\[\mathcal{L}[y](z)=\widehat{y}(z)=\displaystyle\int\limits_{0}^{\infty} e^{-zt} y(t) dt.\]
The inverse Laplace transform has the form
\[\mathcal{L}^{-1} [\widehat{y}](t)=y(t)
=\dfrac{1}{2\pi i}\displaystyle\int_{Br} e^{zt} \widehat{y}(z)\,dz,\]
where $Br = \{z : \operatorname{Re}(z) = \sigma, \, \sigma > 0\}$ is the Bromwich path \cite{AshR1021}.
Let us write the inverse Laplace transform as follows  (see \cite{AshR1023}, p. 254--255)
\begin{equation}\label{Rav701}
y(t)
=\displaystyle\int\limits_{0}^{\infty} e^{-rt} K(r) dr,
\end{equation}
where
\[
K(r)=
-\dfrac{1}{\pi} Im\left\{\widehat{y}(z) \Big|_{z=re^{i \pi}}\right\}
\]

The Laplace convolution of two functions can be written in the form (see \cite{AshR102}, p.19)
\begin{equation}\label{Rav702}
\mathcal{L}^{-1} [\widehat{y}(z) \widehat{f}(z)]
=\mathcal{L}^{-1} [\widehat{y}(z)] \ast \mathcal{L}^{-1} [\widehat{f}(z)](t)
=\displaystyle\int\limits_0^t y(\tau) f(t-\tau)\,d \tau.
\end{equation}

\begin{lemma}\label{Rav201} (see \cite{AshR102}, p.98)
If $0<\rho \leq 1$, then the Laplace transform of the Caputo fractional derivative is given by:
\[
\mathcal{L} [D^{\rho}_{t} y](z)=z^{\rho} \mathcal{L} [y](z)-z^{\rho-1} y(0).
\]
\end{lemma}
\begin{lemma}\label{Rav210}
The Cauchy problem for the Rayleigh--Stokes equation
\begin{equation}\label{Rav211}
\left\{
\begin{array}{ll}
y'(t)+\lambda(1+\gamma D_{t}^{\rho})y(t)=0, \quad \gamma>0, \,\ \lambda>0, \,\ t>0; \\
y(0)=1.
\end{array}
\right.
\end{equation}
has a unique solution
\begin{equation}\label{Rav212}
y(t)= A_{\rho}(\lambda, t)=\frac{\gamma}{\pi} \int\limits_{0}^{\infty}e^{-rt} \frac{\lambda^2 r^{\rho-1}sin \rho \pi}{(-r+\lambda \gamma r^{\rho }cos\rho \pi+\lambda)^{2}+(\lambda\gamma
r^{\rho} sin\rho\pi)^{2}} dr.
\end{equation}
\end{lemma}
\begin{proof} Apply the Laplace transform to equation \eqref{Rav211} to get
\[\mathcal{L}[ y'(t)+\lambda y(t)+\lambda \gamma D_{t}^{\rho}y(t)]=0.\]
By the linearity of the Laplace transform, it follows that:
\[\mathcal{L} [y'](z)+\lambda \mathcal{L} [y] (z)+\lambda \gamma \mathcal{L} [D_{t}^{\rho}y](z)=0.\]
Using Lemma \ref{Rav201}, we obtain:
\[z\mathcal{L} [y] (z)-y(0)+\lambda \mathcal{L} [y] (z)+\lambda \gamma (z^{\rho} \mathcal{L} [y] (z)-z^{\rho-1} y(0))=0.\]
Incorporating the initial condition $y(0)=1$, the equation becomes:
\[z\mathcal{L} [y] (z)-1+\lambda \mathcal{L} [y] (z)+\lambda \gamma z^{\rho} \mathcal{L} [y] (z) - \lambda \gamma z^{\rho-1} =0,\]
and simplifying the above expression yields
\[
\mathcal{L} [y] (z)=
\dfrac{1+\lambda \gamma z^{\rho-1}} {z+\lambda+\lambda \gamma z^{\rho}}.
\]
Taking the inverse Laplace transform of $\mathcal{L} [y] (z)$, we find $y(t)$
\[y(t)=\dfrac{1}{2\pi i} \displaystyle\int_{Br} e^{zt}\dfrac{1+\lambda \gamma z^{\rho-1}} {z+\lambda+\lambda \gamma z^{\rho}}\,dz.\]
Finally, according to the equality \eqref{Rav701}, the spectral density $K(r)$ is given by:
\[
K(r)=
-\dfrac{1}{\pi} Im\left\{\dfrac{1+\lambda \gamma z^{\rho-1}} {z+\lambda+\lambda \gamma z^{\rho}}\Big|_{z=re^{i \pi}}\right\}=
-\dfrac{1}{\pi} Im \left\{\dfrac{1+\lambda \gamma(re^{i \pi})^{\rho-1}}
{re^{i \pi}+\lambda+\lambda \gamma(re^{i \pi})^{\rho}} \right\}
\]
\[
=-\dfrac{1}{\pi} Im \left\{
\dfrac{1+\lambda \gamma r^{\rho-1} (\cos(\rho-1)\pi+i \sin (\rho-1)\pi)}
{-r+\lambda+\lambda \gamma r^{\rho}(\cos \rho \pi+i \sin \rho \pi)}
\right\}
\]
\[
=-\dfrac{1}{\pi} Im \left\{
\dfrac{1-\lambda \gamma r^{\rho-1} \cos \rho \pi-i \lambda \gamma r^{\rho-1}\sin \rho \pi}
{-r+\lambda+\lambda \gamma r^{\rho}\cos\rho \pi+i \lambda \gamma r^{\rho} \sin\rho \pi}
\right\}.
\]
The imaginary part is given by
\[Im \left\{\dfrac{a-ib}{c+id}\right\}=-\dfrac{ad+bc}{c^2+d^2}.\]
Substituting the values, we obtain:
\[
K(r)=\dfrac{1}{\pi} \cdot
\dfrac{(1-\lambda \gamma r^{\rho-1} \cos \rho \pi)(\lambda \gamma r^{\rho} \sin\rho \pi)+(\lambda \gamma r^{\rho-1}\sin \rho \pi)(\lambda-r+\lambda \gamma r^{\rho}\cos\rho \pi)}
{(\lambda-r+\lambda \gamma r^{\rho}\cos\rho \pi)^2+(\lambda \gamma r^{\rho} \sin\rho \pi)^2}
\]
\[
=\dfrac{1}{\pi} \cdot
\dfrac{(\lambda \gamma r^{\rho-1}\sin \rho \pi)(r(1-\lambda \gamma r^{\rho-1} \cos \rho \pi)+\lambda-r+\lambda \gamma r^{\rho}\cos\rho \pi)}
{(\lambda-r+\lambda \gamma r^{\rho}\cos\rho \pi)^2+(\lambda \gamma r^{\rho} \sin\rho \pi)^2}.
\]
Further simplification leads to
\[
K(r)
=\dfrac{\gamma}{\pi}
\dfrac{\lambda^2  r^{\rho-1} \sin \rho \pi}
{(-r+\lambda+\lambda \gamma r^{\rho} \cos\rho \pi)^2+(\lambda \gamma r^{\rho} \sin\rho \pi)^2}.
\]
Consequently, we arrive at the solution given in \eqref{Rav212}.

Thus, the lemma is proved.
\end{proof}

The general solution of the homogeneous and nonhomogeneous Rayleigh--Stokes equations was obtained in \cite{AshR102} (p. 318 and pp. 324--325), where it was expressed in terms of the Wright function. However, this representation is not essential for the purposes of the present work.

\begin{lemma}\label{Rav230}
Let $A_{\rho}(\lambda, t)$ be a solution of the Cauchy problem for the Rayleigh--Stokes equation \eqref{Rav211}. Then
\begin{enumerate}
\item $A_{\rho}(\lambda, 0)=1, \,\ 0<A_{\rho}(\lambda, t)<1, \,\ t>0,$
\item $D_t A_{\rho} (\lambda, t)<0$.
\end{enumerate}
\end{lemma}
\begin{proof}
By the property of the Laplace transform (see, for example, \cite{AshR1022}, p.168), one has
\[y(0)=\lim_{z \rightarrow +\infty}z \mathcal{L}y(z),\]
which yields:
\[
y(0)=
\lim_{z \rightarrow +\infty} z\dfrac{1+\lambda \gamma z^{\rho-1}} {z+\lambda+\lambda \gamma z^{\rho}}
=
\displaystyle\lim_{z \rightarrow +\infty} \dfrac{1+\lambda \gamma z^{\rho-1}} {1+\lambda z^{-1}+\lambda \gamma z^{\rho-1}}=1.
\]
Therefore, $y(0)=A_{\rho} (\lambda, 0)=1.$

Let us differentiate $A_{\rho} (\lambda, t)$ with respect to $t$. Then
\[
D_t A_{\rho} (\lambda, t)
=-\frac{\gamma}{\pi}\int\limits_{0}^{\infty} re^{-rt} \frac{\lambda^2 r^{\rho-1}sin \rho \pi}{(-r+\lambda \gamma r^{\rho }cos\rho \pi+\lambda)^{2}+(\lambda\gamma r^{\rho} sin\rho\pi)^{2}}\,dr<0.
\]
Thus, $A_{\rho} (\lambda, t)$ is strictly positive and monotonically decreasing in $t$, satisfying $0<A_{\rho} (\lambda, t) \leq 1$.
\end{proof}

\begin{lemma}\label{Rav361}
The following estimate holds for all $t \in [0, \, T]$ and $k \geq 1$:
\[A_{\rho}(\lambda_k, t) \geq C(\rho, \gamma, \lambda_1)>0,\]
where
\[C(\rho, \gamma, \lambda_1)
=\dfrac{\gamma \sin \rho \pi}{3 \pi}\displaystyle\int\limits_{0}^{\infty}
\dfrac{r^{\rho-1} e^{-rT}}{\dfrac{r^2}{\lambda_1^2}+\gamma^2 r^{2 \rho}+1}\,dr.
\]
\end{lemma}
\begin{proof}
Let us find a lower bound for $A_{\rho}(\lambda_k, t)$. To this end, we estimate $A_{\rho}(\lambda_k, t)$ from below by bounding the denominator from above. Since  $(a+b+c)^2 \leq 3(a^2+b^2+c^2)$, we obtain
\[
(-r+\lambda_k+\lambda \gamma r^{\rho} \cos\rho \pi)^2+(\lambda \gamma r^{\rho} \sin\rho \pi)^2
\]
\[
\leq 3(r^2+\lambda_k^2+(\lambda_k \gamma r^{\rho} \cos\rho \pi)^2)+(\lambda_k \gamma r^{\rho} \sin\rho \pi)^2
\leq 3\lambda_k^2 \left(\dfrac{r^2}{\lambda_1^2}+\gamma^2 r^{2 \rho}+1\right).
\]
Substituting this estimate into the expression for $A_{\rho}(\lambda_k, t)$, we arrive at
\[
A_{\rho}(\lambda_k, t) \geq \dfrac{\gamma \sin \rho \pi}{3 \pi} \displaystyle\int\limits_{0}^{\infty}
\dfrac{r^{\rho-1} e^{-rT}}{\dfrac{r^2}{\lambda_1^2}+\gamma^2 r^{2 \rho}+1}\,dr.
\]
It remains to verify the convergence of this improper integral. Since the denominator is bounded from below by $1$, we have
\[
\displaystyle\int\limits_{0}^{\infty}
\dfrac{r^{\rho-1} e^{-rT}}{\dfrac{r^2}{\lambda_1^2}+\gamma^2 r^{2 \rho}+1}\,dr
\leq
\displaystyle\int\limits_{0}^{\infty} r^{\rho-1} e^{-rT}\,dr.
\]
By changing variables $\tau=rT$, the latter integral becomes
\[
T^{-\rho} \displaystyle\int\limits_{0}^{\infty} \tau^{\rho-1} e^{-\tau} d \tau=\dfrac{\Gamma(\rho)}{T^{\rho}}.
\]
\end{proof}

\begin{lemma}\label{Rav240}
The Cauchy problem for the Rayleigh--Stokes equation
\begin{equation}\label{Rav250}
\left\{
\begin{array}{ll}
y'(t)+\lambda(1+\gamma D_{t}^{\rho})y(t)=f(t), \quad \gamma>0, \,\ \lambda>0, \,\ t>0; \\
y(0)=y_0,
\end{array}
\right.
\end{equation}
has a unique solution
\begin{equation}\label{Rav251}
y(t)=y_0 A_{\rho}(\lambda, t)+\int\limits_0^t B_{\rho} (\lambda, t-\tau) f(\tau)\,d \tau,
\end{equation}
where
$B_{\rho}(\lambda, t)$ is given as (see \cite{AshR96}):
\begin{equation}\label{Rav252}
B_{\rho}(\lambda, t)
=\frac{\gamma}{\pi}\int\limits_{0}^{\infty}e^{-rt}\frac{\lambda r^{\rho} \sin \rho \pi}
{(-r+\lambda \gamma r^{\rho} \cos\rho \pi+\lambda)^{2}+(\lambda\gamma r^{\rho} sin\rho\pi)^{2}} \,dr.
\end{equation}
\end{lemma}

\begin{proof}
Let us apply the Laplace transform to the equation \eqref{Rav250}:
\[\mathcal{L} [y'(t)+\lambda y(t)+\lambda \gamma D_{t}^{\rho}y(t)]=\mathcal{L} [f(t)].\]
By the linearity property of the Laplace transform, it follows that
\[\mathcal{L} [y'](z)+\lambda \mathcal{L} [y] (z)+\lambda \gamma \mathcal{L} [D_{t}^{\rho}y(t)](z)=\mathcal{L} [f](z).\]
Using Lemma \ref{Rav201}, we obtain
\[z\mathcal{L} [y] (z)-y(0)+\lambda \mathcal{L} [y] (z)+\lambda \gamma (z^{\rho} \mathcal{L} [y] (z)-z^{\rho-1} y(0))=\mathcal{L} [f](z).\]
Thus, the equation takes the following form:
\[z\mathcal{L} [y] (z)-y_0+\lambda \mathcal{L} [y] (z)+\lambda \gamma z^{\rho} \mathcal{L} [y] (z)- \lambda \gamma z^{\rho-1} y_0=\mathcal{L} [f](z),\]
or
\[\mathcal{L} [y] (z)=
y_0 \dfrac{1+\lambda \gamma z^{\rho-1}} {z+\lambda+\lambda \gamma z^{\rho}}
+\dfrac{\mathcal{L} [f](z)}{z+\lambda+\lambda \gamma z^{\rho}}.
\]
Applying the inverse Laplace transform to this expression, we obtain (refer to formula \eqref{Rav212})
\[
\mathcal{L}^{-1}\left[\dfrac{1+\lambda \gamma z^{\rho-1}} {z+\lambda+\lambda \gamma z^{\rho}}\right]
=A_{\rho} (\lambda, t).\]

The inverse Laplace transform of the second term is provided in \cite{AshR96}, namely,
\[
\mathcal{L}^{-1} \left[\dfrac{1}{z+\lambda+\lambda \gamma z^{\rho}}\right]
=B_{\rho}(\lambda, t).
\]
According to the convolution \eqref{Rav702} for the Laplace transform, we have
\[
\mathcal{L}^{-1} \left[\dfrac{\mathcal{L}[f](z)}{z+\lambda+\lambda \gamma z^{\rho}}\right]
=\int\limits_0^t B_{\rho} (\lambda, t-\tau) f(\tau)\,d \tau.
\]
As a result, for the solution to \eqref{Rav250}, we obtain \eqref{Rav251}.

Thus, the lemma is proved.
\end{proof}

The solution of problem \eqref{Rav250} in a particular case for $\lambda = 1$ and $\gamma = a$ (where $a$ is a positive constant) was obtained in \cite{AshR1023} (pp. 253--255). 

\begin{lemma}\label{Rav260} (see \cite{AshR96})
Let $B_{\rho} (\lambda, t)$ be the solution of the Cauchy problem for the Rayleigh--Stokes equation \eqref{Rav250}. Then
\begin{enumerate}
\item $B_{\rho} (\lambda, 0)=1, \,\ 0<B_{\rho} (\lambda, t)<1, \,\ t>0,$
\item $\lambda B_{\rho} (\lambda, t)< C \min\{t^{-1}, \, t^{\rho-1}\}, \,\ t>0,$
\item $\displaystyle\int\limits_0^T B_{\rho} (\lambda, t)\,dt < \dfrac{1}{\lambda}, \,\ T>0.$
\end{enumerate}
\end{lemma}
\begin{lemma}\label{Rav280} The following equalities hold for $A_{\rho}(\lambda, t)$ and $B_{\rho} (\lambda, t)$
\begin{equation}\label{Rav2520}
D_t A_{\rho}(\lambda, t)=-\lambda B_{\rho} (\lambda, t),
\end{equation}
and
\begin{equation}\label{Rav2510}
A_{\rho}(\lambda, t)
=1-\lambda \displaystyle\int\limits_0^t B_{\rho} (\lambda, \tau)\,d \tau.
\end{equation}
\end{lemma}
\begin{proof} 
Differentiating $A_{\rho}(\lambda, t)$ with respect to $t$ yields (see \eqref{Rav252})
\[
D_t A_{\rho}(\lambda, t)=
-\frac{\gamma}{\pi}\int\limits_{0}^{\infty} re^{-rt} \frac{\lambda^2 r^{\rho-1}sin \rho \pi}{(-r+\lambda \gamma r^{\rho }cos\rho \pi+\lambda)^{2}+(\lambda\gamma r^{\rho} sin\rho\pi)^{2}}\,dr=-\lambda B_{\rho} (\lambda, t).
\]
Integrating this equality from $0$ to $t$ and using the fact that $A_\rho(\lambda,0)=1$, we arrive at \eqref{Rav2510}.
\end{proof}

Note that the function $A_{\rho}(\lambda, t)$ is also discussed in \cite{AshR96}.

\begin{lemma}\label{Rav580} (see \cite{AshR270})
The following estimate holds for all $t \in [0, \, T]$ and $k \geq 1$:
\[B_{\rho}(\lambda_k, t) \geq \dfrac{C(\rho, \gamma, \lambda_1)}{\lambda_k},\]
where
\[C(\rho, \gamma, \lambda_1)
=\dfrac{\gamma \sin \rho \pi}{4}\displaystyle\int\limits_{0}^{\infty}
\dfrac{r^{\rho} e^{-rT}}{\dfrac{r^2}{\lambda_1^2}+\gamma^2 r^{2 \rho}+1}\,dr.
\]
\end{lemma}
\begin{corollary}\label{Rav990} The following estimate holds for all $t \in [0, \, T]$ and $k \geq 1$:
\[|A_{\rho} (\lambda_k, t)-1|\geq C(\rho, \gamma, \lambda_1) t.\]
\end{corollary}
\begin{proof}
By Lemmas \ref{Rav280} and \ref{Rav580}, we obtain
\[|A_{\rho} (\lambda_k, t)-1|=
|\lambda_k \displaystyle\int\limits_0^t B_{\rho} (\lambda_k, \tau)\,d \tau |
\geq
\displaystyle\int\limits_0^t 
C(\rho, \gamma, \lambda_1) \,d \tau 
=
C(\rho, \gamma, \lambda_1) t.
\]
\end{proof}
\begin{lemma}\label{Rav270} (see \cite{AshR98})
There exists a constant $C>0$ such that for any $\varepsilon, \, 0<\varepsilon<1$, we have
\[
D_{t} B_{\rho} (\lambda, t)\leq \dfrac{C \lambda^{\varepsilon}}{t^{1-\varepsilon(1-\rho)}}, \,\ t>0.
\]
\end{lemma}

\section{Forward problem}
Let us present the Cauchy problem \eqref{Rav371} using  the following theorem.
\begin{theorem}\label{Rav372}
Let $\varepsilon \in (0, \, 1)$ and $f \in C([0, \, T]; \, D(A^{\varepsilon}))$.
Then the Cauchy problem \eqref{Rav371} has a unique solution for any $\phi \in D(A)$, and this solution has the form
\begin{equation}\label{Rav373}
u(t)=\sum _{k=1}^{\infty } \, A_{\rho}(\lambda_k, t)\phi_{k} v_k
+\sum _{k=1}^{\infty } \, \left[\int\limits_0^t B_{\rho} (\lambda_k, t-\tau) f_k(\tau)\,d \tau\right] \, v_{k}.
\end{equation}
Furthermore, the following coercivity-type estimate holds for some constants $C, C_{\varepsilon}>0$:
\[
\|D_t u(t)\|^2+\|A u(t)\|^{2}+\|A D_t^{\rho} u(t)\|^2
\]
\begin{equation}\label{Rav374}
\leq Ct^{2(\rho-1)}\|\phi\|^2+ \|\phi\|_1^2+C_{\varepsilon} \max_{t \in [0,\,T]} \|f(t)\|_{\varepsilon}^2, \quad 0< t \leq T.
\end{equation}
\end{theorem}
\begin{proof}
Suppose a solution $u(t)$ to the problem \eqref{Rav371} exists. Since the system $\{v_k\}$ is complete in $H$, the solution can be represented in the form:
\begin{equation}\label{Rav610}
u(t)=\sum _{k=1}^{\infty } \, T_{k} (t)\, v_{k}.
\end{equation}
Substituting this expression into \eqref{Rav371}, we arrive at the Cauchy problem for the coefficients $T_k(t)$:
\begin{equation}\label{Rav620}
\left\{
\begin{array}{ll}
T_{k}'(t)+\lambda_k(1+\lambda D_{t}^{\rho})T_{k}(t)=f_{k}(t), \,\ 0<t \leq T; \\
T_{k}(0)=\phi_{k}.
\end{array}
\right.
\end{equation}
According to Lemma \ref{Rav240}, the solution to this problem is as follows:
\[T_k(t)=A_{\rho}(\lambda_k, t) \phi_{k}+\int\limits_0^t B_{\rho} (\lambda_k, t-\tau) f_k(\tau) \,d \tau.\]
Thus, the series \eqref{Rav373} represents a formal solution to the problem \eqref{Rav371}.

The uniqueness of the solution is established by showing that under homogeneous conditions ($\phi_k=0$, $f_k(t)=0$), the solution to the problem \eqref{Rav620} is identically zero. Due to the completeness of $\{v_k\}$, this follows from the explicit form of the function $T_k(t)$.

Let us prove that the function \eqref{Rav373} satisfies all the conditions of Definition \ref{Rav10}. 

Applying Parseval's equality and the first estimate of Lemma \ref{Rav230}, one has
\[
\left\|\sum _{k=1}^{m} \, A_{\rho}(\lambda_k, t) \phi_{k} v_k \right\|^2
= \sum _{k=1}^{m} \, \left| A_{\rho}(\lambda_k, t) \phi_{k} \right|^2
\leq
\sum _{k=1}^{m} \, |\phi_{k}|^2
\leq \|\phi\|^2.
\]
Applying Parseval's equality, the first estimate of Lemma \ref{Rav260}, and H\"older's inequality, we arrive at
\[
\left\| \sum _{k=1}^{m} \, \left[\int\limits_0^t B_{\rho} (\lambda_k, t-\tau) f_k(\tau)\,d \tau\right] \, v_{k} \right\|^2
= \sum _{k=1}^{m} \, \left| \int\limits_0^t B_{\rho} (\lambda_k, t-\tau) f_k(\tau)\,d \tau \right|^2
\]
\[
\leq
\sum _{k=1}^{m} \,
\left(\int\limits_0^t |B_{\rho} (\lambda_k, t-\tau)|^2 \,d \tau \right)
\left(\int\limits_0^t|f_k(\tau)|^2\,d \tau \right)
\]
(from the first estimate of Lemma \ref{Rav260}, one has $|B_{\rho} (\lambda_k, t-\tau)|^2\leq B_{\rho} (\lambda_k, t-\tau)$)
\[
\leq \sum _{k=1}^{m} \, \dfrac{1}{\lambda_k} \int\limits_0^t|f_k(\tau)|^2\,d \tau
\leq \dfrac{1}{\lambda_1} \int\limits_0^t \sum _{k=1}^{m} \, |f_k(\tau)|^2\,d \tau
\leq C_T \max_{0 \leq t \leq T} \|f\|^2.
\]
Thus,
\begin{equation}\label{Rav5101}
\left\| \sum _{k=1}^{m} \, \left[\int\limits_0^t B_{\rho} (\lambda_k, t-\tau) f_k(\tau)\,d \tau\right] \, v_{k} \right\|^2
\leq C_T \max_{0 \leq t \leq T} \|f\|^2.
\end{equation}

Next, consider the action of the operator $A$.
Apply Parseval's equality, the first estimate of Lemma \ref{Rav230}, to get
\[
\left\|\sum _{k=1}^{m} \, \lambda_k A_{\rho}(\lambda_k, t) \phi_{k} v_k \right\|^2
=
\sum _{k=1}^{m} \, \left|\lambda_k A_{\rho}(\lambda_k, t) \phi_{k} \right|^2
\leq
\sum _{k=1}^{m} \, \lambda_k^2 |\phi_{k}|^2
\leq
\|\phi\|_1^2.
\]
Again, using Parseval's equality, the second estimate of Lemma \ref{Rav260}, we obtain
\[
\left\|
\sum _{k=1}^{m} \,
\left[
\int\limits_0^t \lambda_k B_{\rho} (\lambda_k, t-\tau) f_k(\tau)\,d \tau
\right]
\, v_{k}
\right\|^2
\]
\[
=
\sum _{k=1}^{m} \,
\left|
\int\limits_0^t \lambda_k B_{\rho} (\lambda_k, t-\tau) f_k(\tau)\,d \tau
\right|^2
\leq
\left[\left(
C_T \sum _{k=1}^{m} \,\left|
\int\limits_0^T |(T-\tau)^{\rho-1} f_k(\tau)|\,d \tau\right|^2
\right)^{\frac{1}{2}}\right]^2
\]
(apply the generalized Minkowski inequality)
\[
\leq
C_T
\left[
\int\limits_0^T (T-\tau)^{\rho-1} \left( \sum _{k=1}^{m} \, |f_k(\tau)|^2 \right)^{\frac{1}{2}}\,d \tau
\right]^2
\leq
C_T \max_{t \in [0,\,T]} \|f(t)\|^2, \,\ t>0.
\]
Thus,
\begin{equation}\label{Rav5201}
\left\|
\sum _{k=1}^{m} \,
\left[
\int\limits_0^t \lambda_k B_{\rho} (\lambda_k, t-\tau) f_k(\tau)\,d \tau
\right]
\, v_{k}
\right\|^2
\leq
C_T \max_{t \in [0,\,T]} \|f(t)\|^2.
\end{equation}

Now, consider the differentiation with respect to $t$. Applying Parseval's equality, Lemma \ref{Rav280}, and the second estimate of Lemma \ref{Rav260}, we obtain
\[
\left\|
\sum _{k=1}^{m} \, D_t A_{\rho}(\lambda_k, t) \phi_{k} v_k
\right\|^2
=
\left\|
\sum _{k=1}^{m} \, \lambda_k B_{\rho} (\lambda_k, t)\phi_{k} v_k
\right\|^2
=
\sum _{k=1}^{m} \, \left|\lambda_k B_{\rho} (\lambda_k, t)\phi_{k}\right|^2
\]
\[
\leq
C t^{2(\rho-1)} \sum _{k=1}^{m} \, |\phi_{k}|^2 \leq C t^{2(\rho-1)} \|\phi\|^{2}, \,\ t>0.
\]
By applying Parseval's equality, we have
\[
\left\|
\sum _{k=1}^{m} \,
D_t
\left[
\int\limits_0^t B_{\rho} (\lambda_k, t-\tau) f_k(\tau)\,d \tau
\right]
\, v_{k}
\right\|^2
=
\left\|
\sum _{k=1}^{m} \,
\left[
f_k(t) +
\int\limits_0^t D_t B_{\rho} (\lambda_k, t-\tau) f_k(\tau)\,d \tau
\right]
\, v_{k}
\right\|^2
\]
\[
=
\sum _{k=1}^{m} \,
\left|
f_k(t) +
\int\limits_0^t D_t B_{\rho} (\lambda_k, t-\tau) f_k(\tau)\,d \tau
\right|^2
\]
\[
\leq
2 \max_{0 \leq t \leq T} \|f\|^2 +2 \sum _{k=1}^{m} \,
\int\limits_0^t \left| D_t B_{\rho} (\lambda_k, t-\tau) f_k(\tau)\right|^2 \,d \tau
\]
(for the non-homogeneous part, applying the generalized Minkowski inequality and Lemma \ref{Rav270})
\[
\leq
2 \max_{0 \leq t \leq T} \|f\|^2 +
C \left(
\int\limits_0^t \tau^{\varepsilon(1-\rho)-1}
\left(
\sum _{k=1}^{m} \, |\lambda_k^{\varepsilon} f_{k}(\tau)|^2
\right)^{\frac{1}{2}}\,d \tau
\right)^2
\]
\[
\leq
2 \max_{0 \leq t \leq T} \|f\|^2 +
C_{\varepsilon} \max_{t \in [0,\,T]} \|f(t)\|_{\varepsilon}^2.
\]
Hence,
\begin{equation}\label{Rav5301}
\left\|
\sum _{k=1}^{m} \,
D_t
\left[
\int\limits_0^t B_{\rho} (\lambda_k, t-\tau) f_k(\tau)\,d \tau
\right]
\, v_{k}
\right\|^2
\leq
2 \max_{0 \leq t \leq T} \|f\|^2 +
C_{\varepsilon} \max_{t \in [0,\,T]} \|f(t)\|_{\varepsilon}^2.
\end{equation}
Finally, using equation \eqref{Rav370}, we bound the term
\[\|AD_t^{\rho} u(t)\| \leq \|D_t u(t)\|+\|A u(t)\|+\|f(t)\|.\]
Combining the above estimates, we conclude
\[\|AD_t^{\rho} u(t)\|^2 \leq
C t^{2(\rho-1)} \|\phi\|^2+ \|\phi\|_1^2+C_{\varepsilon} \max_{t \in [0,\,T]} \|f(t)\|_{\varepsilon}^2, \,\ t>0.\]
Here we applied  \eqref{Rav570} with $\varepsilon>0$. Therefore, it follows that the series \eqref{Rav373} is indeed a solution to the problem \eqref{Rav371}, which completes the proof of Theorem \ref{Rav372}.
\end{proof}

\section{Auxiliary problems}
The solution of problem \eqref{Rav1} for $\alpha=1$ is obtained by considering two associated auxiliary problems:
\begin{equation}\label{Rav370}
\left\{
\begin{array}{ll}
D_{t} V(t)+A(1+\gamma D_{t}^{\rho}) V(t)=f(t), \quad 0<t \leq T; \\
V(0)=0.
\end{array}
\right.
\end{equation}
and
\begin{equation}\label{Rav380}
\left\{
\begin{array}{ll}
D_{t} W(t)+A(1+\gamma D_{t}^{\rho})W(t)=0, \quad 0<t \leq T; \\
W (T)=W(0)+\psi,
\end{array}
\right.
\end{equation}
where a function $\psi\in H$ is given.

Since problems \eqref{Rav370} and \eqref{Rav380} are special cases of problems \eqref{Rav1}, their solutions are defined in the same manner as in Definition \ref{Rav10}.

If we define $\psi = \varphi - V(T)$ and let $V(t)$ and $W(t)$ be the solutions to their respective sub-problems, then the solution to \eqref{Rav1} is given by $u(t) = W(t) + V(t)$. So, it is sufficient to study the auxiliary problems.

Let us first note that if condition $\phi=0$ holds, then Theorem \ref{Rav372} yields the following result for the auxiliary problem \eqref{Rav370}:
\begin{theorem}\label{Rav115}
Let $\varepsilon \in (0, \, 1)$ and $f \in C([0, \, T]; \, D(A^{\varepsilon}))$. Then the problem \eqref{Rav370} has a unique solution given by:
\begin{equation}\label{Rav390}
V(t)=\sum _{k=1}^{\infty } \,
\left[
\int\limits_0^t B_{\rho} (\lambda_k, t-\tau) f_k(\tau)\,d \tau
\right]
\, v_{k}.
\end{equation}
Furthermore, there exists a constant $C_{\varepsilon}>0$ such that the following coercive inequality is satisfied:
\begin{equation}\label{Rav391}
\|D_{t} V(t)\|^2+\|AV(t)\|^{2}+\|AD_{t}^{\rho} V(t)\|^2
\leq
C_{\varepsilon} \max_{t \in [0,\,T]} \|f(t)\|_{\varepsilon}^2, \quad 0 \leq t \leq T.
\end{equation}
\end{theorem}

Let us present the following theorem for the problem \eqref{Rav380}. 

\begin{theorem}\label{Rav103}
The problem \eqref{Rav380} has a unique solution for any $\psi \in D(A)$ and this solution has the form
\begin{equation}\label{Rav401}
W(t)=\sum _{k=1}^{\infty}
\dfrac{A_{\rho} (\lambda_k, t)}{A_{\rho} (\lambda_k, T)-1} \psi_{k} v_{k}.
\end{equation}
Furthermore, there exists a constant $C_T := C(\rho, \gamma, \lambda_1, T)>0$ such that the following coercive inequality is satisfied:
\begin{equation}\label{Rav410}
\|D_{t} W(t)\|^2+\|AW(t)\|^{2}+\|AD_{t}^{\rho} W(t)\|^2 \leq
C_T \|\psi\|_{1}^{2}, \quad 0 < t \leq T .
\end{equation}
\end{theorem}
\begin{proof}
Let $W(t)$ be a solution to the non-local problem \eqref{Rav380}. By the completeness of the system $\{v_k\}$ in $H$, the solution can be represented as:
\begin{equation}\label{Rav900}
W(t)=\sum _{k=1}^{\infty } \, T_{k} (t)\, v_{k},
\end{equation}
where $T_{k} (t)$ is a solution of the following problem:
\begin{equation}\label{Rav420}
\left\{
\begin{array}{ll}
D_{t} T_{k} (t)+\lambda_{k} T_{k} (t)+\gamma \lambda_{k} D_{t}^{\rho} T_{k} (t)=0; \\
T_{k}(T)=T_{k}(0)+\psi_k.
\end{array}
\right.
\end{equation}
Letting $T_{k}(0)=b_{k}$, Lemma \ref{Rav210}, implies
\begin{equation}\label{Rav910}
T_{k}(t)=b_{k} A_{\rho} (\lambda_k, t).
\end{equation}
Applying the non-local condition from \eqref{Rav420}, we find an equation to find unknown $b_k$:
\[
b_{k} (A_{\rho} (\lambda_k, T)-1)=\psi_{k}.
\]
Since $A_{\rho} (\lambda_k, T) \neq 1$(by the first estimate of Lemma \ref{Rav230}), we obtain the formal solution \eqref{Rav401} to the problem \eqref{Rav380}.

Let us prove the uniqueness of the solution to problem \eqref{Rav420}. The solution to problem \eqref{Rav420} with the condition $\psi_k=0$ has been defined subject to the condition $T_{k} (T)=T_{k} (0)$. According to Lemma \ref{Rav230}, this implies that $b_k \equiv 0$, and consequently $T_k (t)\equiv 0$. By the completeness of the system $\{v_k\}$, it follows that $W(t) \equiv 0$.

Let us obtain an estimate for $W(t)$. Apply  Parseval's equality, the first estimate of Lemma \ref{Rav230}, and corollary \ref{Rav990} to get
\[
\|\sum _{k=1}^{m}
\dfrac{A_{\rho} (\lambda_k, t)}{A_{\rho} (\lambda_k, T)-1} \psi_{k} v_k\|^2
=
\sum _{k=1}^{m}
|
\dfrac{A_{\rho} (\lambda_k, t)}{A_{\rho} (\lambda_k, T)-1} \psi_{k}|^2
\leq
\dfrac{1}{CT^2} \sum_{k=1}^{m} |\psi_{k}|^2
\leq
C_T \|\psi\|^2.
\]

Next, let us obtain an estimate of $D_t W(t)$. Applying Parseval's equality, the second estimate of Lemma \ref{Rav260}, Lemma \ref{Rav280} and corollary \ref{Rav990}, we obtain
\[
\|
\sum_{k=1}^{m}
\dfrac{D_t A_{\rho} (\lambda_k, t)}{A_{\rho} (\lambda_k, T)-1} \psi_{k} v_{k}
\|^2
=
\sum_{k=1}^{m} 
|
\dfrac{\lambda_k B_{\rho} (\lambda_k, t)}{A_{\rho} (\lambda_k, T)-1} \psi_{k}
|^2
\leq \dfrac{1}{CT^2} t^{2(\rho-1)} \sum_{k=1}^{m} |\psi_k |^2
\leq C_T \|\psi\|^{2}.
\]
Consider the action of the operator $A$. Using Parseval's equality,  the first estimate of Lemma \ref{Rav230}), corollary \ref{Rav990} we get
\[
\|
\sum_{k=1}^{m}\dfrac{\lambda_k A_{\rho} (\lambda_k, t)}{A_{\rho} (\lambda_k, T)-1} \psi_{k} v_{k}
\|^2
=
\sum_{k=1}^{m}
|
\dfrac{\lambda_k A_{\rho} (\lambda_k, t)}{A_{\rho} (\lambda_k, T)-1} \psi_{k}
|^2
\leq \dfrac{1}{CT^2} \sum _{k=1}^{m} |\lambda_k \psi_k |^2
\leq C_T\|\psi\|_{1}^{2}.
\]

It follows from equation \eqref{Rav380} that
\[
\|AD_{t}^{\rho}W(t)\|^2 \leq \|D_{t} W(t)\|^2+\|AW(t)\|^{2}\leq C_T \|\psi\|_{1}^{2}.
\]
Theorem \ref{Rav103} is completely proved.
\end{proof}

The following theorem for the non-local problem \eqref{Rav1} is obtained by combining the statements of the previous two theorems.
\begin{theorem}\label{Rav480}
Let $\varepsilon \in (0, \, 1)$, $f \in C([0, \, T]; \, D(A^{\varepsilon}))$. The problem \eqref{Rav1} has a unique solution for any $\varphi \in D(A)$ and this solution has the form
\begin{equation}\label{Rav450}
u(t)=\sum _{k=1}^{\infty } \, \left[\dfrac{A_{\rho} (\lambda_k, t)}{A_{\rho} (\lambda_k, T)-1}
(\varphi_{k}-V_{k}(T))+V_k(t) \right] \, v_{k},
\end{equation}
where
\[
V_k (t)=
\int\limits_0^t B_{\rho} (\lambda_k, t-\tau) f(\tau)\,d \tau.
\]

Furthermore, a coercivity-type estimate holds for certain constants
$C>0$ and $C_{\varepsilon}>0$ :
\[
\|D_{t} u(t)\|^2+\|Au(t)\|^{2}+\|AD_{t}^{\rho} u(t)\|^2
\]
\begin{equation}\label{Rav460}
\leq
C t^{2(\rho-1)}\|\varphi\|^{2}+C\|\varphi\|_{1}^{2}+C_{\varepsilon} \max_{t \in [0,\,T]} \|f(t)\|_{\varepsilon}^2, \quad 0 < t < T.
\end{equation}
\end{theorem}
\begin{proof}
Let $f \in C([0, \, T]; \, D(A^{\varepsilon}))$ for some $\varepsilon \in (0, \, 1)$, and let $\varphi \in D(A)$. As noted above, if we define $\psi = \varphi - V(T) \in H$, where $V(t)$ and $W(t)$ are the solutions of problems \eqref{Rav370} and \eqref{Rav380}, respectively, then the function
$u(t) = V(t) + W(t)$ is a solution of the problem considered \eqref{Rav1}. Hence, the representation given in \eqref{Rav450} provides a unique solution to the problem.

Furthermore, the estimate \eqref{Rav460} follows directly from the estimates \eqref{Rav391} and \eqref{Rav410}.
\end{proof}

\section{Backward problem}
In this section, we study the following backward problem:
\begin{equation}\label{Rav500}
\left\{
\begin{array}{ll}
D_{t} U(t)+A(1+\gamma D_{t}^{\rho})U(t)=f(t), \quad \gamma>0, \quad 0<t<T; \\
U(T)=\psi,
\end{array}
\right.
\end{equation}
where $\psi \in H$ is a given vector. Let us present the following theorem.
\begin{theorem}\label{Rav510}
Let $f(t) \in C([0, \, T]; \, H)$. The problem
\eqref{Rav500} has a unique solution for any $\psi \in D(A)$. Furthermore, there exists a constant $C=C(\rho, \gamma, \lambda_1, t, T)>0$, such that the following estimate holds
\begin{equation}\label{Rav520}
\|D_t U(t)\|^2+\|A U(t)\|^{2}+\|A D_t^{\rho} U(t)\|^2
\leq C(\|\psi\|_{1}^2+\max_{t \in [0, \, T]} \|f(t)\|^2).
\end{equation}
\end{theorem}

\begin{proof}
To find the solution, we determine the initial unknown data $\phi=U(0)$ using the condition $U(T)=\psi$. From the representation \eqref{Rav373}, we have
\[U(T)=\sum _{k=1}^{\infty } \, A_{\rho}(\lambda_k, T)\phi_{k} v_k
+\sum _{k=1}^{\infty } \, \left[\int\limits_0^T B_{\rho} (\lambda_k, T-\tau) f_k(\tau)\,d \tau\right] \, v_{k}=\sum_{k=1}^{\infty } \,\psi_k v_k.\]
Solving for the Fourier coefficients $\phi$, we obtain
\[
\phi_{k}
= \dfrac{1}{A_{\rho}(\lambda_k, T)}
\left(\psi_k - \int\limits_0^T B_{\rho} (\lambda_k, T-\tau) f_k (\tau)\,d \tau\right),
\quad k \geq 1.
\]
Substituting these coefficients back into the solution formula, the unique formal solution of the backward problem is expressed as follows:
\[
U(t)=
\sum _{k=1}^{\infty } \,
\dfrac{A_{\rho}(\lambda_k, t)}{A_{\rho}(\lambda_k, T)}
\left(\psi_k - \int\limits_0^T B_{\rho} (\lambda_k, T-\tau) f_k (\tau)\,d \tau\right) v_k
\]
\[
+\sum _{k=1}^{\infty } \,
\left[
\int\limits_0^t B_{\rho} (\lambda_k, t-\tau) f_k (\tau)\,d \tau
\right]
\, v_{k}.
\]

Let us proceed to the proof of the estimate \eqref{Rav520}. First, let us obtain an estimate for $U(t)$. The stability of the solution is established by estimating the three terms of the series separately.

The estimate of the first series of $U(t)$ follows from Parseval's equality, applying the first estimate of Lemma \ref{Rav230} for the numerator, Lemma \ref{Rav361} to estimate the denominator term $A_{\rho} (\lambda_k, t)$,
\[
\left\|\sum _{k=1}^{m} \, \dfrac{A_{\rho}(\lambda_k, t)\psi_k}{A_{\rho}(\lambda_k, T)}v_{k}\right\|^2
= \sum _{k=1}^{m} \, \left|\dfrac{A_{\rho}(\lambda_k, t)} {A_{\rho}(\lambda_k, T)}\psi_k \right|^2
\leq \dfrac{1}{C(\rho, \gamma, \lambda_1)} \sum _{k=1}^{m} \, |\psi_k|^2
\leq \dfrac{1}{C(\rho, \gamma, \lambda_1)} \|\psi\|^2.
\]

The second series $U(t)$ is estimated again using Parseval's equality, Lemma \ref{Rav230}, Lemma \ref{Rav361} and applying the first estimate of Lemma \ref{Rav260} to the term under the integral:
\[
\left\|\sum _{k=1}^{m} \,
\dfrac{A_{\rho}(\lambda_k, t)}{A_{\rho}(\lambda_k, T)}\left(\int\limits_0^T B_{\rho} (\lambda_k, T-\tau) f_k (\tau)\,d \tau\right) v_k\right\|^2
\]
\[
=
\sum _{k=1}^{m} \,\left|\dfrac{A_{\rho}(\lambda_k, t)}{A_{\rho}(\lambda_k, T)}\int\limits_0^T B_{\rho} (\lambda_k, T-\tau) f_k (\tau)\,d \tau\right|^{2}
\]
\[
\leq
\sum _{k=1}^{m} \, \left|\dfrac{A_{\rho}(\lambda_k, t)}
{A_{\rho}(\lambda_k, T)}\right|^2
\int\limits_0^T |B_{\rho} (\lambda_k, T-\tau)|^2 |f_k(\tau)|^2\,d \tau
\leq \dfrac{C_T}{C(\rho, \gamma, \lambda_1)} \max_{t \in [0, \, T]} \|f(t)\|^2.
\]
The estimate of the third series of $U(t)$ is given by \eqref{Rav5101}.

Next, let us obtain an estimate of $AU(t)$. For the first series $AU(t)$, the definition of the operator $A$, Parseval's equality, the first estimate of Lemma \ref{Rav230} and Lemma \ref{Rav361} give
\[
\left\|\sum _{k=1}^{m} \, \dfrac{\lambda_k A_{\rho}(\lambda_k, t)\psi_k}{A_{\rho}(\lambda_k, T)} v_{k}\right\|^2
=
\sum _{k=1}^{m} \,\left|\dfrac{\lambda_{k}A_{\rho}(\lambda_k, t)\psi_k}{A_{\rho}(\lambda_k, T)}\right|^{2}
\]
\[
\leq
\dfrac{1}{C(\rho, \gamma, \lambda_1)} \sum _{k=1}^{m} \, |\lambda_{k}\psi_k |^2
\leq \dfrac{1}{C(\rho, \gamma, \lambda_1)} \|\psi\|_1^2.
\]

The second series of $AU(t)$ is estimated by using Parseval's equality, applying the first estimate of Lemma \ref{Rav230}, Lemma \ref{Rav361},
and finally the second estimate of Lemma \ref{Rav260} to the term under the integral
\[
\left\|\sum _{k=1}^{m} \, \dfrac{\lambda_k A_{\rho}(\lambda_k, t)}{A_{\rho}(\lambda_k, T)}\left(\int\limits_0^T B_{\rho} (\lambda_k, T-\tau) f_k (\tau)\,d \tau\right) v_k\right\|^2
\]
\[
=
\sum _{k=1}^{m} \, \left| \dfrac{A_{\rho}(\lambda_k, t)}{A_{\rho}(\lambda_k, T)}\int\limits_0^T \lambda _{k}B_{\rho} (\lambda_k, T-\tau) f_k (\tau)\,d \tau\right|^{2}
\]
\[
\leq
\dfrac{C_T}{C(\rho, \gamma, \lambda_1)} \sum _{k=1}^{m} \,
\left| \int\limits_0^T |(T-\tau)^{\rho-1} f_k(\tau)|\,d \tau \right|^2
\]
(apply the generalized Minkowski inequality)
\[
\leq
\dfrac{C_T}{C(\rho, \gamma, \lambda_1)}
\left(
\int\limits_0^T (T-\tau)^{\rho-1} \left( \sum _{k=1}^{m} \, |f_k(\tau)|^2 \right)^{\frac{1}{2}}\,d \tau
\right)^2
\leq
\dfrac{C_T}{C(\rho, \gamma, \lambda_1)} \max_{t \in [0,\,T]} \|f(t)\|^2.
\]
The estimate of the third series of $AU(t)$ is given by \eqref{Rav5201}.

Lastly, let us obtain an estimate of $D_t U(t)$. Applying Parseval's equality, Lemma \ref{Rav280}, Lemma \ref{Rav361}, and the second estimate of Lemma \ref{Rav260} we have
\[
\left\| D_t \sum _{k=1}^{m} \,\dfrac{A_{\rho}(\lambda_k, t)\psi_k}{A_{\rho}(\lambda_k, T)} v_k \right\|^{2}
=
\sum _{k=1}^{m} \, \left| D_t
\dfrac{A_{\rho}(\lambda_k, t)}{A_{\rho}(\lambda_k, T)}
\psi_k
\right|^2
\]
\[
=
\sum _{k=1}^{m} \, \left|
\dfrac{-\lambda_k B_{\rho}(\lambda_k, t)}{A_{\rho}(\lambda_k, T)}
\psi_k
\right|^2
\leq \dfrac{C t^{2(\rho-1)}}{C(\rho, \gamma, \lambda_1)} \|\psi\|^2.
\]
An estimate for the second term is obtained by applying Parseval's equality, Lemma \ref{Rav280}, Lemma \ref{Rav361}, and Lemma \ref{Rav270}:
\[
\left\| D_t \sum _{k=1}^{m} \,\dfrac{A_{\rho}(\lambda_k, t)}{A_{\rho}(\lambda_k, T)}\int\limits_0^T B_{\rho} (\lambda_k, T-\tau) f_k (\tau)\,d \tau\right\|^{2}
\]
\[
=
\sum _{k=1}^{m} \, \left|
\dfrac{-\lambda_k B_{\rho}(\lambda_k, t)}{A_{\rho}(\lambda_k, T)}
\int\limits_0^T B_{\rho} (\lambda_k, T-\tau) f_k (\tau)\,d \tau
\right|^2
\]
\[
\leq
C t^{2(\rho-1)} \sum _{k=1}^{m} \,
\left(
\int\limits_0^T |(T-\tau)^{\rho-1} f_k(\tau)|\,d \tau
\right)^2
\]
(apply the generalized Minkowski inequality)
\[
\leq
C_T
\left(
\int\limits_0^T (T-\tau)^{\rho-1} \left( \sum _{k=1}^{m} \, |f_k(\tau)|^2 \right)^{\frac{1}{2}}\,d \tau
\right)^2
\leq
C_T \max_{t \in [0,\,T]} \|f(t)\|^2, \,\ t>0.
\]
As a result, the estimate of the third series of $D_tU(t)$ is given by \eqref{Rav5301}.

Therefore, it follows from the above
\[
\|AD_{t}^{\rho}U(t)\|^2 \leq \|D_{t} U(t)\|^2+||AU(t)||^{2}
\leq C(\|\psi\|_1^2+\max_{t \in [0,\,T]} \|f(t)\|^2).
\]
Hence, we arrive at the estimate in \eqref{Rav520}. This completes the proof of the theorem.
\end{proof}

\section{Conclusions}

In this paper, we investigate initial-boundary value, non-local, and backward problems for the Rayleigh--Stokes equation with a fractional Caputo derivative. Previously, the Rayleigh--Stokes equation was studied only with the Riemann--Liouville derivative. The choice of the Caputo derivative is justified primarily by the fact that the Rayleigh-Stokes equation involves a first-order derivative, and if such a derivative exists, then the Caputo derivative certainly exists. On the other hand, the Caputo derivative is a more natural generalization of ordinary derivatives.

Using the Fourier method, we find conditions on the problem data for each of the three problems that guarantee both the existence and uniqueness of the problems under consideration. In particular, it is shown that for the backward problem (i.e., $\alpha = 0$), it is necessary to require much more stringent conditions from the problem data than for the preodic problem (i.e., $\alpha = 1$).

\end{document}